\documentclass{amsart}
\usepackage{amssymb,amsthm, amsmath, amsfonts}
\usepackage{graphics}
\usepackage{hyperref}
\usepackage[all]{xy}
\usepackage{enumerate}
\usepackage{mathrsfs}

\theoremstyle{plain}
\newtheorem{theorem}{Theorem}[section]
\newtheorem{lemma}[theorem]{Lemma}
\newtheorem{proposition}[theorem]{Proposition}
\newtheorem{corollary}[theorem]{Corollary}
\numberwithin{equation}{section}

\theoremstyle{definition}

\newtheorem{definition}[theorem]{Definition}
\newtheorem{example}[theorem]{Example}
\newtheorem{remark}[theorem]{Remark}

\DeclareMathOperator*{\colim}{colim}
\DeclareMathOperator{\End}{End}
\DeclareMathOperator{\Hom}{Hom}
\DeclareMathOperator{\Ob}{Ob}
\DeclareMathOperator{\im}{Im}
\DeclareMathOperator{\op}{op}
\DeclareMathOperator{\prd}{prd}

\newcommand{\Z}{{\mathbb{Z}}}
\newcommand{\cA}{{\mathcal{A}}}

\newcommand{\cC}{{\mathcal{C}}}

\newcommand{\FI}{{\mathcal{FI}}}
\newcommand{\FS}{{\mathcal{FS}}}
\newcommand{\OI}{{\mathcal{OI}}}

\newcommand{\VI}{{\mathcal{VI}}}

\newcommand{\wA}{{\widetilde{\mathcal{A}}}}
\newcommand{\wI}{{\widetilde{I}}}
\newcommand{\wgamma}{{\widetilde{\gamma}}}
\newcommand{\wdelta}{{\widetilde{\delta}}}

\title{On central stability}

\author{Wee Liang Gan}
\address{Department of Mathematics, University of California, Riverside, CA 92521, USA}
\email{wlgan@math.ucr.edu}

\author{Liping Li}
\address{Key Laboratory of Performance Computing and Stochastic Information Processing (Ministry of Education), College of Mathematics and Computer Science, Hunan Normal University, Changsha, Hunan 410081, China}
\email{lipingli@hunnu.edu.cn}

\thanks{The second author is supported by the National Natural Science Foundation of China 11541002, the Construct Program of the Key Discipline in Hunan Province, and the Start-Up Funds of Hunan Normal University 830122-0037. Both authors would like to thank the anonymous referee for carefully checking the manuscript and providing quite a few valuable comments.}


\begin{document}

\begin{abstract}
The notion of central stability was first formulated for sequences of representations of the symmetric groups by Putman. A categorical reformulation was subsequently given by Church, Ellenberg, Farb, and Nagpal using the notion of FI-modules, where FI is the category of finite sets and injective maps. We extend the notion of central stability from FI to a wide class of categories, and prove that a module is presented in finite degrees if and only if it is centrally stable. We also introduce the notion of $d$-step central stability, and prove that if the ideal of relations of a category is generated in degrees at most $d$, then every module presented in finite degrees is $d$-step centrally stable.
\end{abstract}

\maketitle

\section{Introduction}

Let $R$ be a commutative ring. For any small category $\cC$, a \emph{$\cC$-module} over $R$ is, by definition,  a (covariant) functor from $\cC$ to the category of $R$-modules. A \emph{morphism} of $\cC$-modules over $R$ is, by definition, a natural transformation of functors. The category of $\cC$-modules over $R$ is an abelian category. One can define, in a natural way, the notion of a \emph{finitely generated} $\cC$-module over $R$.

Let $\FI$ be the category of finite sets and injective maps, and denote by $\Z_+$ the set of non-negative integers. For any $\FI$-module $V$ over $R$ and $n\in\Z_+$, we write $V_n$ for $V([n])$, where $[n]:=\{1,\ldots,n\}$. Since the automorphism group of $[n]$ in the category $\FI$ is the symmetric group $S_n$, an $\FI$-module $V$ over $R$ gives rise to a sequence $\{V_n\}$ where $V_n$ is a representation of $S_n$. It was shown by Church, Ellenberg, Farb, and Nagpal (see \cite{CEF} and \cite{CEFN}) that many interesting sequences of representations of symmetric groups arise in this way from finitely generated $\FI$-modules. A principle  result they proved is that any finitely generated $\FI$-module over a commutative noetherian ring is noetherian. They deduced as a consequence that if $V$ is a finitely generated $\FI$-module over a commutative noetherian ring, then the sequence $\{V_n\}$ admits an inductive description, in the sense that for all sufficiently large integer $N$, one has
\begin{equation} \label{colimit}
 V_n \cong \colim_{\substack{S\subset [n] \\ |S|\leqslant N}} V(S) \quad \mbox{ for each } n\in \Z_+,
\end{equation}
where the colimit is taken over the poset of all subsets $S$ of $[n]$ such that $|S|\leqslant N$.

To provide the context for our present article, let us briefly describe the way in which \eqref{colimit} is proved in \cite{CEFN}. Suppose that $V$ is a finitely generated $\FI$-module over a noetherian ring. For any finite set $T$, there is a canonical homomorphism of $R$-modules
\begin{equation} \label{canonical map}
\colim_{\substack{S\subset T \\ |S|\leqslant N}} V(S) \longrightarrow V(T).
\end{equation}
Since $V$ is finitely generated, there exists an integer $N'$ such that $V$ is generated by $\sqcup_{i\leqslant N'} V_i$. It is easy to see that the homomorphism \eqref{canonical map}  is surjective if $N\geqslant N'$. To prove the injectivity of \eqref{canonical map} when $N$ is sufficiently large, Church, Ellenberg, Farb, and Nagpal constructed a (Koszul) complex $\widetilde{S}_{-*}V$ of $\FI$-modules which has the property that:
\begin{gather*}
(H_1(\widetilde{S}_{-*} V))(T) = \mathrm{Ker}\Big( \colim_{\substack{S\subset T \\ |T|-2\leqslant |S|\leqslant|T|-1 }} V(S) \to V(T)  \Big).
\end{gather*}
They verified that one has:
\begin{equation}  \label{central stabilization on lhs}
\colim_{\substack{S\subset T \\ |T|-2\leqslant |S|\leqslant|T|-1 }} V(S) = \colim_{\substack{S\subset T \\  |S|<|T| }} V(S).
\end{equation}
From the noetherian property of $V$, they proved that there exists an integer $N''$ such that $(H_1(\widetilde{S}_{-*} V))(T)=0$ if $|T|\geqslant N''$. As a consequence,
\begin{equation} \label{third isomorphism}
\colim_{\substack{S\subset T \\ |S|< |T| }} V(S) = V(T) \quad \mbox{ if } |T| > \max\{N',N''\}.
\end{equation}
Set $N=\max\{N',N''\}$. The homomorphism in \eqref{canonical map} is an isomorphism if $|T|\leqslant N$, for $T$ is terminal in the poset $\{S \mid S\subset T\}$. It now follows by an induction on $|T|$ that \eqref{canonical map} is an isomorphism too if $|T|>N$, for:
\begin{equation*}
\colim_{\substack{S\subset T \\ |S|\leqslant N}} V(S)= \colim_{\substack{U\subset T \\ |U|<|T| } } \, \colim_{\substack{S\subset U \\ |S|\leqslant N} } V(S) = \colim_{\substack{U\subset T \\ |U|<|T|} } V(U) = V(T),
\end{equation*}
where the first isomorphism is routine, the second isomorphism is by the induction hypothesis, and the third isomorphism is by \eqref{third isomorphism}.

It was subsequently proved by Church and Ellenberg that for any $\FI$-module $V$ over an arbitrary commutative ring, if $V$ is presented in finite degrees (in a suitable sense), then there are integers $N'$ and $N''$ such that $V$ is generated by $\sqcup_{i\leqslant N'} V_i$, and $(H_1(\widetilde{S}_{-*} V))(T)=0$ if $|T|\geqslant N''$. Consequently, using the same arguments as above, they showed that the isomorphism \eqref{colimit} holds for all $N$ sufficiently large. This extends the result of their joint work with Farb and Nagpal to $\FI$-modules which are not necessarily noetherian (since every finitely generated $\FI$-module over a commutative noetherian ring is presented in finite degrees).

The left-hand side of \eqref{central stabilization on lhs} is isomorphic to the central stabilization construction of Putman \cite[\S1]{Putman}. In their joint work, Putman and Sam generalized the isomorphism in \eqref{colimit} to modules over complemented categories with a generator of which $\FI$ is an example. A complemented category with a generator $X$ is the data of a symmetric monoidal category and an object $X$ satisfying a list of axioms. Their proof is similar to the one for $\FI$ described above.

A goal of our present paper is to give a very simple and transparent proof of a generalization of the above results. In particular, our proof does not require consideration of the complex $\widetilde{S}_{-*}V$ or similar complexes. Moreover, the setting for our generalization is much simpler than the one of complemented categories with a generator used by Putman and Sam \cite{PS}, and include many more examples; our generalization is, in fact, motivated by the fact that several combinatorial categories studied by Sam and Snowden in \cite{SS-Grobner} do not fall within the framework of \cite{PS}. Another goal of our paper is to explain the role of the quadratic property of $\FI$ in the isomorphism (\ref{central stabilization on lhs}).

\subsection*{Outline of the paper}
This paper is organized as follows.

In Section \ref{generalities}, we define our generalization of the notion of central stability and introduce the notion of $d$-step central stability. We show that in the special case of complemented categories with a generator studied by Putman and Sam, our notion of central stability is indeed equivalent to their notion of central stability. We also show that for the category $\FI$, it is equivalent to the inductive description \eqref{colimit}.

In Section \ref{central stability}, we give a reminder of a key lemma from Morita theory. We then prove our first main result, that a module is presented in finite degrees if and only if it is centrally stable. We deduce that if every finitely generated module is neotherian, then every finitely generated module is centrally stable. We prove that the converse of the preceding statement holds under a local finiteness assumption on the category when the base ring is a commutative noetherian ring. To illustrate the wide applicability of our results, we recall examples of combinatorial categories introduced by Sam and Snowden in \cite{SS-Grobner} which fall within our framework but are not complemented categories with a generator.

In Section \ref{d-step section}, we prove our second main result, that if the ideal of relations of a category is generated in degrees at most $d$, then every module presented in finite degrees is $d$-step centrally stable; for example, if the category is quadratic, then every module presented in finite degrees is $2$-step centrally stable.

To apply our second main result, one needs to have a practical way to check that the ideal of relations of a category is generated in degrees at most $d$. In Section \ref{last section}, we give sufficiency conditions which allow one to do this.

\section{Generalities} \label{generalities}

\subsection{Notations and definitions} \label{definitions subsection}
We prefer to formulate our main results in the more familiar language of modules over algebras. Throughout this paper, we denote by $R$ a commutative ring and $\Z_+$ the set of non-negative integers.

Let $\cA$ be an $R$-linear category, i.e. a category enriched over the category of $R$-modules. We assume that $\Ob(\cA)=\Z_+$ and $\Hom_{\cA}(m,n)=0$ if $m>n$. We set
\begin{equation}\label{category algebra}
A := \bigoplus_{m,n \in \Ob(\cA)} \Hom_{\cA}(m,n).
\end{equation}
There is a natural structure of a (non-unital) $R$-algebra on $A$. (In the terminology of \cite{BP}, $A$ is a \emph{$\Z$-algebra}.) We call $A$ the \emph{category algebra} of $\cA$. For each $n\in \Ob(\cA)$, we denote by $e_n$ the identity endomorphism of $n$. For any $m, n\in \Ob(\cA)$ with $m\leqslant n$, we set
\begin{equation*}
e_{m,n} := e_m+e_{m+1}+ \cdots+ e_n \in A.
\end{equation*}
One has: $Ae_{m,n} = Ae_m \oplus Ae_{m+1} \oplus \cdots \oplus Ae_n$.

A \emph{graded $A$-module} is an $A$-module $V$ such that $V = \bigoplus_{n\in \Ob(\cA)} e_n V$. Observe that if $V$ is a graded $A$-module, then any $A$-submodule of $V$ is also a graded $A$-module.

A graded $A$-module $V$ is \emph{finitely generated} if it is finitely generated as an $A$-module. Equivalently, $V$ is finitely generated if for some $N\in \Ob(\cA)$, there is an exact sequence $\bigoplus_{i\in I} Ae_{0,N} \to V \to 0$ where the indexing set $I$ is finite. A graded $A$-module $V$ is \emph{noetherian} if every $A$-submodule of $V$ is finitely generated.

A graded $A$-module $V$ is \emph{finitely presented} if for some $N\in \Ob(\cA)$, there is an exact sequence of the form
\begin{equation*}
\bigoplus_{j\in J} Ae_{0,N} \longrightarrow \bigoplus_{i\in I} Ae_{0,N} \longrightarrow V \longrightarrow 0
\end{equation*}
where both the indexing sets $I$ and $J$ are finite.

A graded $A$-module $V$ is \emph{presented in finite degrees} if for some $N\in \Ob(\cA)$, there is an exact sequence
\begin{equation} \label{presentation in finite degree exact sequence}
\bigoplus_{j\in J} Ae_{0,N} \longrightarrow \bigoplus_{i\in I} Ae_{0,N} \longrightarrow V \longrightarrow 0
\end{equation}
where the indexing sets $I$ and $J$ may be finite or infinite. The smallest $N$ for which such an exact sequence exists is called the \textit{presentation degree} of $V$, and we denote it by $\prd(V)$.

\begin{remark} \label{remark on presentation exact sequence}
It is easy to see that if $V$ is a graded $A$-module presented in finite degrees, then for every $N\geqslant \prd(V)$, there exists an exact sequence of the form \eqref{presentation in finite degree exact sequence}, where $I$ and $J$ may depend on $N$.
\end{remark}

The main definitions of this paper are as follows.

\begin{definition} \label{definition of central stability}
A graded $A$-module $V$ is called \emph{centrally stable} if for all $N$ sufficiently large, one has
\begin{equation} \label{central stability isomorphism}
 Ae  \otimes_{e A e} eV \cong V \quad \mbox{ where } e=e_{0,N}.
\end{equation}
\end{definition}

\begin{definition} \label{definition of d-step centrally stable}
Let $d$ be an integer $\geqslant 1$. A graded $A$-module $V$ is called \emph{$d$-step centrally stable} if for all $N$ sufficiently large, one has
\begin{equation*}
 Ae  \otimes_{e A e} eV \cong \bigoplus_{n\geqslant N-(d-1)} e_n V \quad \mbox{ where } e=e_{N-(d-1),N}.
\end{equation*}
\end{definition}

\subsection{Inductive descriptions}
We now explain the relation of our definition of central stability given above to the notion of central stability given by Putman and Sam in \cite{PS} for complemented categories with a generator, and the relation to the inductive description \eqref{colimit} in the special case of $\FI$-modules as given by Church, Farb, Ellenberg, and Nagpal \cite{CE, CEFN}.

Let $\cC$ be a small category such that $\Ob(\cC)=\Z_+$, and $\Hom_{\cC}(m,n)=\emptyset$ if $m>n$. We denote by $\cA _{\cC}$ the $R$-linear category with $\Ob(\cA_{\cC})=\Z_+$ and $\Hom_{\cA_{\cC}}(m,n)$ the free $R$-module with basis $\Hom_{\cC}(m,n)$ for each $m,n\in \Z_+$. Let $A_{\cC}$ be the category algebra of $\cA_{\cC}$; see \eqref{category algebra}.

Recall that a $\cC$-module over $R$ is, by definition, a (covariant) functor from $\cC$ to the category of $R$-modules. For any $\cC$-module $V$ over $R$ and $n\in\Z_+$, we write $V_n$ for $V(n)$. If $V$ is a $\cC$-module over $R$, then $\bigoplus_{n\in\Ob(\cC)} V_n$ is a graded $A_{\cC}$-module. This defines an equivalence from the category of $\cC$-modules over $R$ to the category of graded $A_{\cC}$-modules. We say that a $\cC$-module $V$ over $R$ has a certain property (such as centrally stable) if the graded $A_{\cC}$-module $\bigoplus_{n\in\Ob(\cC)} V_n$ has the property.

For any $M, N\in \Ob(\cC)$ with $M\leqslant N$, we write $\cC_{M,N}$ for the full subcategory of $\cC$ on the set of objects $\{M, M+1, \ldots, N\}$. Then the category of $\cC_{M,N}$-modules over $R$ is equivalent to the category of $e A_{\cC} e$-modules where $e=e_{M, N}$. Let $\iota_{M,N}:\cC_{M,N} \to \cC$ be the inclusion functor. We define
the \emph{restriction} functor $\mathrm{Res}_{M,N}$ along $\iota_{M,N}$ by
\begin{align*}
\mathrm{Res}_{M,N}: \mbox{(category of $\cC$-modules over $R$)} &\longrightarrow \mbox{(category of $\cC_{M,N}$-modules over $R$)}, \\
V &\longmapsto V\circ \iota_{M,N}.
\end{align*}
The \emph{left Kan extension functor} $\mathrm{Lan}_{M,N}$ along $\iota_{M,N}$ is a left adjoint functor to $\mathrm{Res}_{M,N}$; for every $\cC_{M,N}$-module $V$ over $R$ and $n\in \Ob(\cC)$, one has:
\begin{equation} \label{Lan formula}
(\mathrm{Lan}_{M,N}V)_n = \colim_{\substack{\alpha: s\to n \\ M\leqslant s\leqslant N }} V_s,
\end{equation}
where $V_s=V(s)$ and the colimit is taken over the (comma) category whose objects are the morphisms $\alpha : s\to n$ in $\cC$ such that $M\leqslant s\leqslant N$; see \cite[Theorem 2.3.3]{KS} or \cite[Section X.3, (10)]{Mac}.

The following proposition gives a reformulation for the notions of central stability and $d$-step central stability for $\cC$-modules over $R$.

\begin{proposition} \label{colimit formulation of central stability}
Let $V$ be a $\cC$-module over $R$. For every $M,N \in \Ob(\cC)$ with $M\leqslant N$, one has:
\begin{equation} \label{LanRes}
 \bigoplus_{n\in\Ob(\cC)} \left( \mathrm{Lan}_{M,N} (\mathrm{Res}_{M,N} V) \right)_n \cong A e \otimes_{eAe} e\Big( \bigoplus_{n\in\Ob(\cC)} V_n \Big) \quad \mbox{ where } A=A_{\cC} \mbox{ and } e = e_{M,N}.
\end{equation}
In particular, $V$ is centrally stable if and only if for all sufficiently large $N$, one has
\begin{equation*}
V_n \cong \colim_{\substack{\alpha: s\to n \\ s\leqslant N }} V_s  \quad \mbox{ for each }n\in \Z_+.
\end{equation*}
Moreover, $V$ is $d$-step centrally stable if and only if for all sufficiently large $N$, one has
\begin{equation*}
V_n \cong \colim_{\substack{\alpha: s\to n \\ N-(d-1)\leqslant s\leqslant N }} V_s \quad \mbox{ for each }n\geqslant N-(d-1).
\end{equation*}
\end{proposition}
\begin{proof}
Let $e=e_{M,N}$. For any $\cC$-module $V$ over $R$, one has
\begin{equation*}\bigoplus_{M\leqslant n\leqslant N} \left( \mathrm{Res}_{M,N} V\right)_n = e\Big( \bigoplus_{n\in\Ob(\cC)} V_n \Big).
\end{equation*}
Since the functor:
\begin{align*}
\mbox{(category of $eAe$-modules)} &\longrightarrow \mbox{(category of graded $A$-modules)}, \\ W &\longmapsto Ae\otimes_{eAe} W,
\end{align*}
is left adjoint to the functor:
\begin{align*}
\mbox{(category of graded $A$-modules)} &\longrightarrow \mbox{(category of $eAe$-modules)}, \\ V\longmapsto eV,
\end{align*}
it follows that for any $\cC_{M,N}$-module $W$ over $R$, one has
\begin{equation*}
 \bigoplus_{M\leqslant n\leqslant N} \left( \mathrm{Lan}_{M,N} W \right)_n \cong  A e \otimes_{eAe} \Big( \bigoplus_{M\leqslant n\leqslant N} W_n \Big).
\end{equation*}
We have proven the isomorphism \eqref{LanRes}. The remaining statements now follow from the formula \eqref{Lan formula}.
\end{proof}

A complemented category with a generator, in the sense of Putman and Sam \cite{PS}, is the data of a symmetric monoidal category $\mathtt{A}$ and an object $X$ of $\mathtt{A}$ satisfying a list of axioms. We do not recall those axioms here since we will not need them; however, a consequence of the axioms is that the full subcategory $\cC$ of $\mathtt{A}$ on the set of objects $X^n$ for $n\in\Z_+$ is a skeleton of $\mathtt{A}$, and $\Hom_{\cC}(X^m, X^n)=\emptyset$ if $m>n$. We may identify the set of objects of $\cC$ with $\Z_+$ in the obvious way. Following Putman and Sam \cite[Theorem E]{PS}, an $\mathtt{A}$-module $V$ over $R$ is \emph{centrally stable} if for all sufficiently large $N$, the functor $V$ is the left Kan extension of the restriction of $V$ to the full subcategory of $\mathtt{A}$ spanned by the objects isomorphic to $X^n$ for some $n\leqslant N$.

\begin{corollary}
Let $\mathtt{A}$ be a complemented category with a generator $X$. Suppose that $\cC$ is the skeleton of $\mathtt{A}$ spanned by the objects $X^n$ for all $n\in \Z_+$. Let $V$ be an $\mathtt{A}$-module over $R$, and regard $V$ also as a $\cC$-module by restriction along the inclusion functor $\cC \to \mathtt{A}$. Then $V$ is centrally stable in the sense of Putman and Sam \cite[Theorem E]{PS} if and only if $V$ is centrally stable as a $\cC$-module over $R$.
\end{corollary}
\begin{proof}
This is immediate from \eqref{LanRes}.
\end{proof}

The full subcategory of $\FI$ spanned by the objects $[n]$ for all $n\in \Z_+$ is a skeleton of $\FI$. In the following corollary, we identify the set of objects of this skeleton with $\Z_+$ in the obvious way.

\begin{corollary} \label{inductive central stability of FI-modules}
Suppose that $\cC$ is the skeleton of $\FI$ spanned by the objects $[n]$ for all $n\in \Z_+$. Let $V$ be an $\FI$-module over $R$, and regard $V$ also as a $\cC$-module by restriction along the inclusion functor $\cC\to \FI$. Then $V$ is centrally stable as a $\cC$-module over $R$ if and only if for all sufficiently large $N$, the isomorphism \eqref{colimit} holds.
\end{corollary}
\begin{proof}
By Proposition \ref{colimit formulation of central stability}, it suffices to show that
\begin{equation*}
\colim_{\substack{\alpha: s\to n \\ s\leqslant N }} V_s \cong \colim_{\substack{S\subset [n] \\ |S|\leqslant N}} V(S).
\end{equation*}
But this is immediate from the observation that the natural functor:
\begin{align*}
\{ \alpha \mid \alpha \in \Hom_{\FI}([s],[n]) \mbox{ and } s\leqslant N \} &\longrightarrow \{ S \mid S\subset [n] \mbox{ and } |S|\leqslant N \},\\
\alpha &\longmapsto \mathrm{Im}(\alpha),
\end{align*}
is final; see \cite[Theorem IX.3.1]{Mac}. (Some authors refer to final functors as cofinal functors. See, for example, \cite[Definition 2.5.1]{KS}.)
\end{proof}

Let us also mention that for a principal ideal domain $R$, Dwyer defined the notion of a \emph{central coefficient system} $\rho$ in \cite{Dw} to mean a sequence $\rho_n$ of $\mathrm{GL}_n(R)$-modules and maps $F_n: \rho_n \to \rho_{n+1}$ such that $F_n$ is a $\mathrm{GL}_n(R)$-map (when $\rho_{n+1}$ is considered as a $\mathrm{GL}_n(R)$-module by restriction) and the image of $F_{n+1}F_n$ is invariant under the action of the permutation matrix $s_{n+2}\in \mathrm{GL}_{n+2}(R)$ which interchanges the last two standard basis vectors of $R^{n+2}$. Suppose that $\cC$ is the skeleton of $\FI$ spanned by the objects $[n]$ for all $n\in \Z_+$. A central coefficient system $\rho$ defines a $\cC$-module $V$ over $R$ with $V_n = \rho_n$ and such that the standard inclusion $[n]\hookrightarrow [n+1]$ induces the map $F_n$. If $V$ is finitely generated as a $\cC$-module over $R$, then by \cite[Theorem C]{CEFN} it is a centrally stable $\cC$-module over $R$.

\section{Central stability} \label{central stability}

\subsection{Key lemma} The following lemma is a standard result in Morita theory; see, for instance, \cite[Theorem 6.4.1]{Cohn}. We recall its proof here since it plays a key role in the proof of our main results.

\begin{lemma} \label{main lemma}
Let $A$ be any (possibly non-unital) ring, and $e$ an idempotent element of $A$. If $V$ is an $A$-module such that, for some indexing sets $I$ and $J$, there is an exact sequence
\begin{equation} \label{finite presentation}
\bigoplus_{j\in J} Ae \longrightarrow \bigoplus_{i\in I} Ae \longrightarrow V \longrightarrow 0,
\end{equation}
then $Ae \otimes_{eAe} eV \cong V$.
\end{lemma}

\begin{proof}
Applying the right-exact functor $Ae\otimes_{eAe} e(-)$ to (\ref{finite presentation}), we obtain the first row in the following commuting diagram:
\begin{equation} \label{diagram in main lemma}
\xymatrix{
\bigoplus_{j\in J} Ae \otimes_{eAe} eAe \ar[r]\ar[d] & \bigoplus_{i\in I} Ae \otimes_{eAe} eAe \ar[r]\ar[d] & Ae\otimes_{eAe} eV \ar[r]\ar[d] & 0 \\
\bigoplus_{j\in J} Ae \ar[r] & \bigoplus_{i\in I} Ae \ar[r] & V \ar[r] & 0
}
\end{equation}
Both the rows in (\ref{diagram in main lemma}) are exact. Since the two leftmost vertical maps are isomorphisms, the rightmost vertical map is also an isomorphism.
\end{proof}

\subsection{Central stability and presentation in finite degrees}
Let $\cA$ be an $R$-linear category such that $\Ob(\cA)=\Z_+$ and $\Hom_{\cA}(m,n)=0$ if $m>n$. Denote by $A$ the category algebra of $\cA$; see \eqref{category algebra}.

\begin{theorem} \label{general centrally stable}
Let $V$ be a graded $A$-module. Then $V$ is presented in finite degrees if and only if it is centrally stable. Moreover, if $V$ is presented in finite degrees, then the isomorphism \eqref{central stability isomorphism} holds if and only if $N\geqslant \prd(V)$.
\end{theorem}
\begin{proof}
Suppose that $V$ is presented in finite degrees, and $N\geqslant \prd(V)$. Then there is an exact sequence \eqref{presentation in finite degree exact sequence} with $e=e_{0,N}$; see Remark \ref{remark on presentation exact sequence}. Hence, by Lemma \ref{main lemma}, one has $Ae\otimes_{eAe} eV \cong V$.

Conversely, suppose that $N$ is an integer such that $Ae \otimes_{eAe} eV \cong V$, where $e=e_{0,N}$. There is an exact sequence $\bigoplus_{j\in J} eAe \to \bigoplus_{i\in I} eAe \to eV \to 0$ for some indexing sets $I$ and $J$. Applying the functor $Ae\otimes_{eAe} (-)$, we obtain an exact sequence of the form \eqref{presentation in finite degree exact sequence} with $e=e_{0,N}$. Therefore, $V$ is presented in finite degrees and $N \geqslant \prd(V)$.
\end{proof}

\begin{corollary} \label{noetherian case}
If every finitely generated graded $A$-module is noetherian, then every finitely generated graded $A$-module is centrally stable.
\end{corollary}
\begin{proof}
If $V$ is a finitely generated graded $A$-module, then $V$ is finitely presented, hence centrally stable by Theorem \ref{general centrally stable}.
\end{proof}

The above proofs are very simple in comparison to the proofs for the special cases given in \cite[Theorem B]{CE}, \cite[Theorem C]{CEFN}, and \cite[Theorem E]{PS}. Let us emphasize, however, that the crux is to recognize that the notion of central stability in \cite{CE, CEFN, PS} can be reformulated in Morita theory; \emph{this was not a priori obvious}.

\subsection{Central stability and noetherian property}
We prove a partial converse to Corollary \ref{noetherian case}.

Let $\cA$ be an $R$-linear category such that $\Ob(\cA)=\Z_+$ and $\Hom_{\cA}(m,n)=0$ if $m>n$. Denote by $A$ the category algebra of $\cA$; see \eqref{category algebra}. We say that $A$ is \emph{locally finite} if $\Hom_{\cA}(m,n)$ is a finitely generated $R$-module for all $m, n \in \Z_+$.

\begin{corollary}
Suppose that $R$ is a commutative noetherian ring and $A$ is locally finite. If every finitely generated graded $A$-module is centrally stable, then every finitely generated graded $A$-module is noetherian.
\end{corollary}
\begin{proof}
Let $V$ be a finitely generated graded $A$-module and $U$ a (graded) $A$-submodule of $V$. We want to prove that $U$ is finitely generated as a graded $A$-module. Since $V/U$ is a finitely generated graded $A$-module, it is centrally stable, hence presented in finite degrees by Theorem \ref{general centrally stable}. Thus, for some $N\in \Ob(\cA)$, there is a commuting diagram
\begin{equation*}
\xymatrix{
& \bigoplus_{j\in J} Ae_{0,N} \ar[r] \ar[d]^-{g} & \bigoplus_{i\in I} Ae_{0,N} \ar[r] \ar[d]^-{f} & V/U \ar[r] \ar@{=}[d] & 0 \\
0 \ar[r] & U \ar[r] & V \ar[r] & V/U \ar[r] & 0.
}
\end{equation*}
It suffices to prove that both $\mathrm{Im}(g)$ and  $\mathrm{Coker} (g)$ are finitely generated graded $A$-modules.

By the Snake Lemma, there is an isomorphism $\mathrm{Coker} (g) \cong \mathrm{Coker} (f)$, so $\mathrm{Coker} (g)$ is a finitely generated graded $A$-module.

Since $A$ is locally finite and $V$ is finitely generated as a graded $A$-module, it follows that $e_{0,N}V$ is a finitely generated $R$-module, hence $e_{0,N}V$ is a noetherian $R$-module. But $\mathrm{Im}(g)$ is generated as a graded $A$-module by $e_{0,N}\mathrm{Im}(g)$, and $e_{0,N}\mathrm{Im}(g)$ is an $R$-submodule of the noetherian $R$-module $e_{0,N}V$, hence $\mathrm{Im}(g)$ is finitely generated as a graded $A$-module.
\end{proof}

\subsection{Examples} \label{subsection examples}
Let us mention some examples of combinatorial categories introduced by Sam and Snowden \cite{SS-Grobner} which fall within our framework. The central stability of modules presented in finite degrees for these categories follow from Theorem \ref{general centrally stable} (and were not previously known).

In the following examples, the set of objects of the category $\cC$ is $\Z_+$. For any $n\in \Z_+$, we set $[n]:=\{1,\ldots, n\}$.

(i) $\cC=\FI_a$ where $a$ is an integer $\geqslant 1$. For any $m, n \in \Z_+$, a morphism $m\to n$ in $\FI_a$ is a pair $(f,c)$ where $f:[m]\to [n]$ is an injective map and $c:[n]\setminus \mathrm{Im}(f) \to [a]$ is any map. The composition of $(f_1,c_1) : m\to n$ and $(f_2,c_2): n\to \ell$ is defined to be $(f_2\circ f_1, c):m\to \ell$ where $c(i)=c_1(j)$ if $i=f_2(j)$ for some $j\in [n]\setminus \mathrm{Im}(f_1)$, and $c(i)=c_2(i)$ if $i\in [\ell]\setminus \mathrm{Im}(f_2)$.

(ii) $\cC=\OI_a$ where $a$ is an integer $\geqslant 1$. For any $m, n \in \Z_+$, a morphism $m\to n$ in $\OI_a$ is a pair $(f,c)$ where $f:[m]\to [n]$ is a strictly increasing map and $c:[n]\setminus \mathrm{Im}(f) \to [a]$ is any map. The composition of morphisms is defined in the same way as above for $\FI_a$.

(iii) $\cC=\FS^{\op}$, the opposite of the category $\FS$. For any $m,n \in \Z_+$, a morphism $n\to m$ in $\FS$ is a surjective map $[n]\to [m]$. The composition of morphisms in $\FS$ is defined to be the composition of maps.

It was proved by Sam and Snowden \cite{SS-Grobner} that for the categories $\cC$ in (i)-(iii) above, every finitely generated graded $\cC$-module over a commuatative noetherian ring is noetherian (and hence they are centrally stable by Corollary \ref{noetherian case}).

\section{{\it d}-step central stability} \label{d-step section}

\subsection{Ideal of relations}
Suppose $V$ is a finitely generated $\FI$-module over a commutative noetherian ring. The isomorphism \eqref{colimit} says that, provided $N$ is sufficiently large, $V$ can be recovered from its restriction to the full subcategory of $\FI$ on the set of objects $\{S \in \Ob(\FI) \mid |S| \leqslant N\}$; using the isomorphism \eqref{central stabilization on lhs}, we see that in fact we can recover $V(S)$ for $|S|>N$ using only the restriction of $V$ to the full subcategory of $\FI$ on the set of objects $\{S \in \Ob(\FI) \mid N-1\leqslant |S| \leqslant N\}$. The purpose of the present section and the next is to show that this is a consequence of the quadratic property of $\FI$. More generally, we shall prove that if the ideal of relations of a category algebra $A$ is generated in degrees $\leqslant d$ (in the sense of Definition \ref{relations degree definition} below), then any graded $A$-module which is presented in finite degrees is $d$-step centrally stable.

Let $\cA$ be an $R$-linear category such that $\Ob(\cA)=\Z_+$ and $\Hom_{\cA}(m,n)=0$ if $m>n$. Denote by $A$ the category algebra of $\cA$; see \eqref{category algebra}. For any $m\in\Z_+$, let
\begin{equation*}
\wA(m,m) := \End_{\cA} (m);
\end{equation*}
for any $n\geqslant m+1$, let
\begin{equation*}
\wA(m,n) := \Hom_{\cA}(n-1,n) \otimes_{\End_{\cA}(n-1)} \cdots \otimes_{\End_{\cA}(m+1)} \Hom_{\cA}(m,m+1).
\end{equation*}
We define a $R$-linear category $\wA$ with $\Ob(\wA)=\Z_+$ by $\Hom_{\wA}(m,n)=\wA(m,n)$ if $m\leqslant n$, and $\Hom_{\wA}(m,n)=0$ if $m>n$. There is a natural $R$-linear functor from $\wA$ to $\cA$ which is the identity map on the set of objects. For any $n\geqslant m$, let $\wI(m,n)$ be the kernel of the map $\wA(m,n) \to \Hom_{\cA}(m,n)$. One has $\wI(m,n)=0$ whenever $n$ is $m$ or $m+1$.

\begin{definition} \label{relations degree definition}
Let $d$ be an integer $\geqslant 1$. We say that \emph{the ideal of relations of $A$ is generated in degrees $\leqslant d$} if, whenever $n\geqslant m+2$, the map $\wA(m,n) \to \Hom_{\cA}(m,n)$ is surjective, and whenever $n\geqslant m+d$, one has:
\begin{equation}\label{ideal}
\wI(m,n) = \sum_{r=m}^{n-d}  \wA(r+d, n)\otimes_{\End_{\cA}(r+d)} \wI(r,r+d)\otimes_{\End_{\cA}(r)} \wA(m,r).
\end{equation}
\end{definition}

\begin{remark}
If $d=1$, so that the ideal of relations of $A$ is generated in degrees $\leqslant 1$, then $\wI(m,n)=0$ for every $m$ and $n$, and so $\wA$ and $\cA$ are the same. In most interesting examples, one has $d\geqslant 2$.
\end{remark}

\begin{remark}
If $d=2$, so that the ideal of relation of $A$ is generated in degrees $\leqslant 2$, then $A$ is a quadratic algebra whose degree $k$ component is $\displaystyle{\bigoplus_{m\in \Z_+}} \Hom_{\cA}(m,m+k)$ for each $k\in \Z_+$. Many combinatorial categories such as $\FI_a$, $\OI_a$ and $\FS^{\op}$ are quadratic; see Section \ref{last section} below.
\end{remark}

\begin{example} \label{plactic monoid}
Fix a finite totally ordered set $\Omega$. Recall (see \cite{LLT}) that the plactic monoid $M$ on $\Omega$ is the monoid generated by $\Omega$ with defining relations
\begin{gather*}
xzy = zxy \quad \mbox{ if } x\leqslant y <z,\\
yxz = yzx \quad \mbox{ if } x< y\leqslant z.
\end{gather*}
An element $w\in M$ is said to be of length $\ell(w)=n$ if $w$ is a product of $n$ elements of $\Omega$; it is clear that $\ell(w)$ is well-defined. Now define $\cC$ to be the category with $\Ob(\cC)=\Z_+$ and
\begin{equation*}
\Hom_{\cC}(m,n) = \{ w\in M \mid \ell(w) = n-m \}.
\end{equation*}
The composition of morphisms in $\cC$ is given by the product in $M$. Then the category algebra $A_{\cC}$ is not quadratic but has ideal of relations generated in degrees $\leqslant 3$.
\end{example}

It is plain that the preceding example can be generalized to any monoid with a presentation whose defining relations do not change the length of words.

\subsection{{\it d}-step central stability}

The following theorem is the second main result of this paper.

\begin{theorem} \label{d-step centrally stable}
Let $d$ be an integer $\geqslant 1$. Suppose the ideal of relations of $A$ is generated in degrees $\leqslant d$. If a graded $A$-module $V$ is presented in finite degrees, then $V$ is $d$-step centrally stable.
\end{theorem}

To prove Theorem \ref{d-step centrally stable}, we need the following lemma.

\begin{lemma} \label{reducing idempotent}
Suppose the ideal of relations of $A$ is generated in degrees $\leqslant d$. If $V$ is a graded $A$-module, and $n>N \geqslant m+d$, then
\begin{equation*}
e_n A f \otimes_{fAf} fV \cong e_n A e \otimes_{eAe} eV \quad \mbox{ where } e=e_{m,N},\quad f=e_{m+1,N}.
\end{equation*}
\end{lemma}
\begin{proof}
Observe that
\begin{equation*}
e_n A e = e_n A e_m \oplus e_n A f, \qquad eV = e_m V \oplus fV.
\end{equation*}
We have a natural map
\begin{equation*}
\Phi: e_n A f \otimes_{fAf} fV \to e_n A e \otimes_{eAe} eV.
\end{equation*}
We shall construct a map $\Psi$ inverse to $\Phi$. First, define the maps
\begin{equation*}
\xymatrix{
\wI(m, n) \otimes_{R} e_m V \ar[d]^{\mu} & \\
\Hom_{\cA}(m+1,n)\otimes_{\End_{\cA}(m+1)} \Hom_{\cA}(m,m+1) \otimes_{R} e_m V  \ar[r]^{\hspace{1.2in}\widetilde{\theta}} \ar[d]^{\nu} & e_n A f \otimes_{fAf} fV \\
\Hom_{\cA}(m,n) \otimes_{R} e_m V &
}
\end{equation*}
where $\mu$ and $\nu$ are the obvious maps defined using composition of morphisms, and $\widetilde{\theta}$ is defined by $\widetilde{\theta}(\alpha\otimes \beta\otimes x) = \alpha \otimes \beta x$. The map $\nu$ is surjective and its kernel is the image of $\mu$. Using (\ref{ideal}) and $\Hom_{\cA}(m+1,m+d) \subset fAf$, we see that $\widetilde{\theta}\mu=0$. Therefore, $\widetilde{\theta}$ factors uniquely through $\nu$ to give a map
\begin{equation*}
\theta : \Hom_{\cA}(m,n) \otimes_{R} e_m V \to e_n A f \otimes_{fAf} fV.
\end{equation*}

Now define
\begin{equation*}
\widetilde{\Psi} : e_n A e \otimes_{R} eV \to A f \otimes_{fAf} fV
\end{equation*}
by
\begin{equation*}
\widetilde{\Psi} (\alpha \otimes x) = \left\{ \begin{array}{ll}
\theta(\alpha\otimes x) & \mbox{ if } \alpha\in e_nAe_m \mbox{ and } x\in e_m V,\\
\alpha\otimes  x & \mbox{ if } \alpha\in e_n A f \mbox{ and } x\in fV,\\
0 & \mbox{ otherwise }.
\end{array} \right.
\end{equation*}
It is plain that $\widetilde{\Psi}$ descends to a map
\begin{equation*}
\Psi : e_n A e \otimes_{eAe} eV \to A f \otimes_{fAf} fV,
\end{equation*}
and $\Psi$ is an inverse to $\Phi$.
\end{proof}

We can now prove Theorem \ref{d-step centrally stable}.

\begin{proof}[Proof of Theorem \ref{d-step centrally stable}]
By Theorem \ref{general centrally stable}, for all $N$ sufficiently large, one has
\begin{equation*}
  A e_{0,N} \otimes_{e_{0,N}Ae_{0,N}} e_{0,N}V \cong  V;
\end{equation*}
in particular, for each $n\in \Z_+$, one has $e_n A e_{0,N} \otimes_{e_{0,N}Ae_{0,N}} e_{0,N}V \cong  e_n V$.

If $n>N\geqslant d-1$, then by Lemma \ref{reducing idempotent},
\begin{align*}
e_n A e_{N-(d-1),N} \otimes_{e_{N-(d-1),N}Ae_{N-(d-1),N}} e_{N-(d-1),N}V
&\cong e_n A e_{N-d,N} \otimes_{e_{N-d,N}Ae_{N-d,N}} e_{N-d,N}V\\
&\hspace{6pt}\vdots\\
&\cong e_n A e_{0,N} \otimes_{e_{0,N}Ae_{0,N}} e_{0,N}V
\end{align*}

If $N-(d-1) \leqslant n \leqslant N$, the map
\begin{equation*}
e_n A e_{N-(d-1),N} \otimes_{e_{N-(d-1),N}Ae_{N-(d-1),N}} e_{N-(d-1),N}V \to e_n V, \quad \alpha\otimes x \mapsto \alpha x
\end{equation*}
has an inverse defined by $x \mapsto e_n \otimes x$.
\end{proof}

\section{Sufficiency conditions} \label{last section}

\subsection{Sufficiency conditions}
To apply Theorem \ref{d-step centrally stable}, we need to be able to check if the ideal of relations of $A$ is generated in degrees $\leqslant d$. In this section, we provide sufficiency conditions which allow one to check this.

Let $\cC$ be a small category such that $\Ob(\cC)=\Z_+$, and $\Hom_{\cC}(m,n)=\emptyset$ if $m>n$. Recall that $\cA_{\cC}$ is the $R$-linear category with $\Ob(\cA_{\cC})=\Z_+$ and $\Hom_{\cA_{\cC}}(m,n)$ the free $R$-module with basis $\Hom_{\cC}(m,n)$ for each $m,n\in \Z_+$. Let $A_{\cC}$ be the category algebra of $\cA_{\cC}$; see \eqref{category algebra}.

\begin{definition}
We say that the \emph{ideal of relations of $\cC$ is generated in degrees $\leqslant d$} if the ideal of relations of $A_{\cC}$ is generated in degrees $\leqslant d$.
\end{definition}

\begin{proposition} \label{combinatorial condition}
Suppose $d\geqslant 2$. The ideal of relations of $\cC$ is generated in degrees $\leqslant d$ if the following two conditions are satisfied:

(i) The composition map $\Hom_{\cC}(l, n) \times \Hom_{\cC}(m,l)\to \Hom_{\cC}(m,n)$ is surjective whenever $m<l<n$.

(ii) For every $\alpha_1, \alpha_2 \in \Hom_{\cC}(m+1,n)$ and $\beta_1, \beta_2\in \Hom_{\cC}(m,m+1)$ satisfying
\begin{equation*}
\alpha_1\beta_1 = \alpha_2\beta_2 \quad \mbox{ and } \quad n>m+d,
\end{equation*}
there exists $\gamma\in \Hom_{\cC}(m+d,n)$ and $\delta_1, \delta_2 \in Hom_{\cC}(m+1,m+d)$ such that the following diagram commutes:
\begin{equation}  \label{combinatorial condition diagram}
\xymatrix{
 & m+1 \ar[dr]_{\delta_1} \ar[drrrr]^{\alpha_1} & & & & \\
m \ar[ur]^{\beta_1} \ar[dr]_{\beta_2} & & m+d  \ar[rrr]^{\hspace{-1cm}\gamma} & & & n \\
 & m+1 \ar[ur]^{\delta_2} \ar[urrrr]_{\alpha_2} & & & & \\
}
\end{equation}
\end{proposition}

\begin{proof}
Condition (i) implies that $\wA_{\cC} (m,n)\to \Hom_{\cA_{\cC}} (m,n)$ is surjective if $n\geqslant m+2$. We shall prove, by induction on $n-m$, that (\ref{ideal}) holds. The case $n-m=d$ is trivial. Now suppose $n-m>d$.
Let
\begin{equation*}
\wA_{\cC}(m,n) \stackrel{\pi_1}{\longrightarrow} \Hom_{\cA_{\cC}}(m+1,n)\otimes_{\End_{\cA_{\cC}}(m+1)}  \Hom_{\cA_{\cC}} (m,m+1)
\stackrel{\pi_2}{\longrightarrow} \Hom_{\cA_{\cC}}(m,n)
\end{equation*}
be the obvious maps defined by composition of morphisms. It is easy to see (and can be proved in the same way as \cite[Lemma 6.1]{Li}) that $\wI(m,n)$ is spanned over $R$ by elements of the form
\begin{equation} \label{element in kernel}
\xi_{n-1}\otimes \cdots\otimes \xi_m - \xi'_{n-1}\otimes \cdots\otimes \xi'_m
\end{equation}
such that $\xi_i, \xi'_i \in \Hom_{\cC}(i,i+1)$ for $i=m,\ldots, n-1$ and $\xi_{n-1}\cdots\xi_m = \xi'_{n-1}\cdots\xi'_m$. By condition (ii), there exists $\gamma\in \Hom_{\cC}(m+d,n)$ and $\delta_1, \delta_2 \in \Hom_{\cC}(m+1,m+d)$ such that
\begin{equation*}
\xi_{n-1}\cdots\xi_{m+1}=\gamma \delta_1, \qquad \xi'_{n-1}\cdots\xi'_{m+1}=\gamma\delta_2, \qquad \delta_1\xi_m=\delta_2\xi'_m.
\end{equation*}
We can choose
\begin{equation*}
\wgamma = \wgamma_{n-1}\otimes \cdots \otimes \wgamma_{m+d}\in \wA_{\cC}(m+d,n),
\end{equation*}
where $\wgamma_i \in \Hom_{\cC}(i,i+1)$  for $i=m+d,\ldots, n-1$, such that $\wgamma_{n-1} \cdots\wgamma_{m+d} = \gamma$. We can also choose
\begin{equation*}
\wdelta  = \wdelta_{m+d-1}\otimes \cdots \otimes \wdelta_{m+1} \in \wA_{\cC}(m+1,m+d),
\end{equation*}
where $ \wdelta_i \in \Hom_{\cC}(i,i+1)$ for $i=m+1,\ldots, m+d-1$, such that $\wdelta_{m+d-1} \cdots\wdelta_{m+1} = \delta_1$. Similarly, choose
\begin{equation*}
\wdelta'  = \wdelta'_{m+d-1}\otimes \cdots \otimes \wdelta'_{m+1} \in \wA_{\cC}(m+1,m+d),
\end{equation*}
where $ \wdelta'_i \in \Hom_{\cC}(i,i+1)$ for $i=m+1,\ldots, m+d-1$, such that $\wdelta'_{m+d-1} \cdots\wdelta'_{m+1} = \delta_2$.
The element in (\ref{element in kernel}) can be written as:
\begin{equation*}
( \xi_{n-1}\otimes \cdots \otimes \xi_{m+1} - \wgamma\otimes \wdelta   )\otimes \xi_m
 - ( \xi'_{n-1}\otimes \cdots \otimes \xi'_{m+1} - \wgamma\otimes \wdelta'   )\otimes \xi'_m
 + \wgamma \otimes (\wdelta \otimes \xi_m - \wdelta' \otimes \xi'_m).
\end{equation*}
Since
\begin{gather*}
\xi_{n-1}\otimes \cdots \otimes \xi_{m+1} - \wgamma\otimes \wdelta \in \wI(m+1, n),\\
\xi'_{n-1}\otimes \cdots \otimes \xi'_{m+1} - \wgamma\otimes \wdelta'  \in \wI(m+1,n),\\
\wdelta \otimes \xi_m - \wdelta' \otimes \xi'_m\in \wI(m,m+d),
\end{gather*}
the result follows by induction.
\end{proof}

\begin{corollary} \label{sufficiency conditions}
Suppose that $\cC$ satisfies the two conditions in Proposition \ref{combinatorial condition} for some $d\geqslant 2$. If a $\cC$-module $V$ is presented in finite degrees, then $V$ is $d$-step centrally stable.
\end{corollary}
\begin{proof}
By Proposition \ref{combinatorial condition}, the ideal of relations of $\cC$ is generated in degrees $\leqslant d$. Hence, we may apply Theorem \ref{d-step centrally stable}.
\end{proof}

\begin{remark}
Suppose that $\cC$ satisfies condition (i) in Proposition \ref{combinatorial condition}, and its ideal of relations is generated in degrees $\leqslant d$. In this case, condition (ii) might not hold, so it is not a necessary condition. For example, suppose that $d=2$, and one has
\begin{gather*}
\Hom_{\cC}(0,1)=\{ \beta_1, \beta_2, \beta_3 \},\quad
\Hom_{\cC}(1,2)=\{ \beta'_1, \beta'_2, \beta'_3, \beta'_4 \},\quad
\Hom_{\cC}(2,3)=\{ \beta''_1, \beta''_2\},
\end{gather*}
with the defining relations
\begin{gather*}
\beta'_1 \beta_1 = \beta'_3 \beta_3,\quad
\beta'_2 \beta_2 = \beta'_4 \beta_3,\quad
\beta''_1 \beta'_3 = \beta''_2 \beta'_4,
\end{gather*}
as depicted in the following diagram:
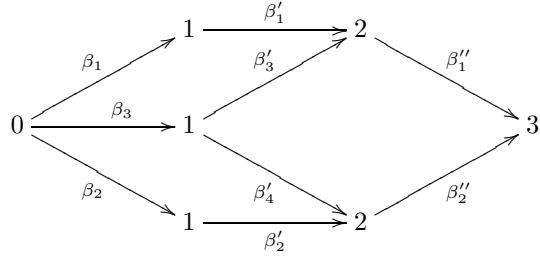
\begin{equation*}
\xymatrix{
 && 1 \ar[rr]^{\beta'_1} && 2 \ar[drr]^{\beta''_1} &&  \\
0 \ar[urr]^{\beta_1} \ar[drr]_{\beta_2} \ar[rr]^{\hspace{5mm}\beta_3} && 1 \ar[urr]^{\beta'_3} \ar[drr]_{\beta'_4} &&  &&  3 \\
 && 1 \ar[rr]_{\beta'_2} && 2 \ar[urr]_{\beta''_2} && \\
}
\end{equation*}
Let $\alpha_1 = \beta''_1\beta'_1$ and $\alpha_2 = \beta''_2\beta'_2$. Then one has
\begin{equation*}
\alpha_1 \beta_1 = \beta''_1\beta'_1 \beta_1 = \beta''_1 \beta'_3\beta_3 = \beta''_2\beta'_4\beta_3 = \beta''_2\beta'_2\beta_2 = \alpha_2 \beta_2.
\end{equation*}
On the other hand, it is easy to see that there do not exist $\gamma$, $\delta_1$ and $\delta_2$ such that the diagram in \eqref{combinatorial condition diagram} (with $d=2$, $m=0$, $n=3$) commutes.
\end{remark}

\subsection{Examples}
Using Proposition \ref{combinatorial condition}, one can easily check, for example, that (a skeleton of) $\FI$ is quadratic. Let $\cC$ be the full subcategory of $\FI$ on the set of objects $\{[n] \mid n\in\Z_+\}$. It is easy to see that $\cC$ satisfies condition (i) in Proposition \ref{combinatorial condition}. We need to verify condition (ii). Thus, suppose that we have injective maps:
\begin{equation*}
\alpha_1, \alpha_2 : [m+1] \longrightarrow [n] \quad \mbox{ and } \quad \beta_1, \beta_2 : [m] \longrightarrow [m+1],
\end{equation*}
such that $\alpha_1 \circ \beta_1 = \alpha_2 \circ \beta_2$ and $n\geqslant m+2$. Since
\begin{equation*}
|\im(\alpha_1)\cup\im(\alpha_2)| = |\im(\alpha_1)| + |\im(\alpha_2)| - |\im(\alpha_1)\cap \im(\alpha_2)| \leqslant (m+1) + (m+1) - m = m+2,
\end{equation*}
we can choose $S\subset [n]$ such that $\im(\alpha_1)\cup\im(\alpha_2)\subset S$ and $|S|=m+2$. Let $\gamma: [m+2]\to [n]$ be any injective map whose image is $S$. Then there exists a unique map $\delta_1$ (respectively $\delta_2$) such that $\gamma\circ \delta_1 = \alpha_1$ (respectively $\gamma\circ \delta_2 = \alpha_2$). Clearly, $\delta_1$ and $\delta_2$ are injective. We have:
\[ \gamma\circ \delta_1\circ \beta_1 = \alpha_1\circ \beta_1 = \alpha_2\circ \beta_2 = \gamma\circ \delta_2 \circ \beta_2. \]
Since $\gamma$ is injective, it follows that $\delta_1\circ \beta_1 = \delta_2 \circ \beta_2$. Therefore, it follows from Proposition \ref{combinatorial condition} that $\cC$ (or more precisely, the algebra $A_{\cC}$) is quadratic.

Similarly, one can apply Proposition \ref{combinatorial condition} to show that the categories $\FI_a$, $\OI_a$, $\FS^{\op}$, and $\VI(\mathbb{F})$ are quadratic, where for any field $\mathbb{F}$, the set of objects of $\VI(\mathbb{F})$ is $\Z_+$ and the morphisms $m\to n$ are the injective linear maps $\mathbb{F}^m \to \mathbb{F}^n$.

\begin{remark}
A twisted commutative algebra $E$ is an associative unital graded algebra $E=\bigoplus_{n\geqslant 0} E_n$ where each $E_n$ is equipped with the structure of an $S_n$-module such that the multiplication map $E_n\otimes E_m \to E_{n+m}$ is $S_n\times S_m$-equivariant, and $yx=\tau(xy)$ for every $x\in E_n$, $y\in E_m$ (where $\tau\in S_{n+m}$ switches the first $n$ and last $m$ elements of $\{1,\ldots,n+m\}$); see \cite[\S 8.1.2]{SS-tca}. Any twisted commutative algebra $E$ gives rise to an $R$-linear category $\cA$ with $\Ob(\cA)=\Z_+$ and
\begin{equation*}
\Hom_{\cA} (m,n) = RS_n \otimes_{RS_{n-m}} E_{n-m}.
\end{equation*}
In particular, when $E$ is the twisted commutative algebra with $E_n$ the trivial $S_n$-module for every $n$ (and with the obvious multiplication map), the $R$-linear category $\cA$ we obtained is precisely the category $\cA_{\cC}$ associated to a skeleton $\cC$ of $\FI$. One might ask if the category algebra of every category $\cA$ obtained from a twisted commutative algebra is quadratic. This is not true. For example, let $d$ be any integer $>2$, and let $E$ be the twisted commutative algebra with $E_n$ the trivial $S_n$-module if $n<d$, otherwise let $E_n=0$; the multiplication map of $E$ is defined in the obvious way. Then $E$ gives rise to an $R$-linear category $\cA$ with the property that the ideal of relations of its category algebra is generated in degrees $\leqslant d$, but the category algebra is not quadratic.
\end{remark}

\end{document}